\documentclass[a4paper]{amsart}

\usepackage{amsthm}
\usepackage{amsmath}
\usepackage{amssymb}
\usepackage{latexsym}
\usepackage{enumerate}
\usepackage{graphicx}
\usepackage[utf8]{inputenc}
\usepackage{epsfig} 
\usepackage{pgf}
\usepackage{tikz}
\usepackage{float}
\usepackage{dsfont}
\usepackage{times}

\newtheorem{theoremA}{Theorem}

\newtheorem{theorem}{Theorem}[section]
\newtheorem{lemma}[theorem]{Lemma}

\newtheorem{proposition}[theorem]{Proposition}
\newtheorem{corollary}[theorem]{Corollary}

\theoremstyle{definition}
\newtheorem{definition}[theorem]{Definition}

\newtheorem{remark}[theorem]{Remark}

\numberwithin{equation}{section}

\newcommand{\Vol}{\operatorname{Vol}}

\newcommand{\erre}{\mathbb{R}}

\newcommand{\kan}{\mathbb{K}^{n}(b)}

\newcommand{\Han}{\mathbb{H}^n(b)}
\newcommand{\Hatwo}{\mathbb{H}^2(b)}

\newcommand{\gr}{\operatorname{\nabla}}
\newcommand{\Sup}{\operatorname{Sup}}
\begin{document}
\title[Chern-Osserman Inequality for surfaces]{Chern-Osserman Inequality for minimal surfaces in a
Cartan-Hadamard manifold with strictly negative sectional curvatures}
\author[A. Esteve]{Antonio Esteve*}
\address{I.E.S. Alfonso VIII, Cuenca-Departamento de Matem\'{a}ticas, Universidad de
Castilla la Mancha, Cuenca, Spain.}
\email{aesteve7@gmail.com}
\author[V. Palmer]{Vicente Palmer**}
\address{Departament de Matem\`{a}tiques- Institute of New Imaging Technologies, Universitat Jaume I, Castellon, Spain.}
\email{palmer@mat.uji.es}
\thanks{}
\thanks{* Work partially supported by DGI grant MTM2010-21206-C02-02.\\ ** Work partially supported by the Caixa Castell\'{o} Foundation, and DGI grant MTM2010-21206-C02-02.}
\subjclass[2000]{Primary 53A15, 53C20; Secondary 53C42}
\keywords{minimal surface, Chern-Osserman Inequality, Gauss-Bonnet theorem,
Hessian-Index comparison theory, extrinsic balls.}
\maketitle

\begin{abstract}
We state and prove a Chern-Osserman-type Inequality in terms of the volume growth for minimal surfaces $S$ which have finite total extrinsic curvature and are properly immersed in a Cartan-Hadamard manifold $N$ with sectional curvatures bounded from above by a negative quantity $K_N \leq b <0$ and such that they are not too curved (on average) with respect to the Hyperbolic space with constant sectional curvature given by the upper bound $b$. We have also proven the same Chern-Osserman-type Inequality for minimal surfaces with finite total extrinsic curvature and properly immersed in an asymptotically hyperbolic Cartan-Hadamard manifold $N$ with sectional curvatures bounded from above by a negative quantity $K_N \leq b <0$.

\end{abstract}


\section{Introduction and main results}\label{secIntro}

In the papers \cite{Ch2} and \cite{Che3}, a Chern-Osserman type inequality was studied for a completely, properly and minimally immersed surface (cmi for short) in the Hyperbolic space, extending the classical result originally established by S.S. Chern and R. Osserman in \cite{ChOss} for cmi surfaces in the Euclidean space to this strictly negatively curved setting. 

Chern-Osserman's result (in fact, an improvement on this result due to M.T. Anderson in \cite{A1} and to L.P. Jorge and W.H. Meeks in \cite{JM}, see also White's work \cite{W} for an approach to this problem for non-minimal surfaces in the Euclidean space) relates the Euler characteristic $\chi(S)$ of a cmi surface with finite total curvature in $\erre^n$ with this total curvature and the (finite) supremum of the (non-decreasing) volume growth of the extrinsic domains (known as the {\em extrinsic balls}) $E_r=S^2\cap B^{0,n}_r$. We denote as $B^{b,n}_r$ the geodesic $r$-ball in $\kan$, which is the simply connected real space form with constant sectional curvature $b$. We also denote as $S^{b,n-1}_r$ the geodesic $r$-sphere in $\kan$.
 We have 
\begin{equation}\label{ChernOssEq}
-\chi(S)= \frac{1}{4\pi}\int_S\Vert B^S\Vert^2 d\sigma -\Sup_{r}\frac{\Vol(S^2\cap B^{0,n}_r)}{\Vol(B^{0,2}_r)}
\end{equation}

In contrast to what happens with cmi surfaces in $\erre^n$, the total Gaussian curvature of surfaces $S^2$ immersed in the hyperbolic space $\Han$ is always infinite, by the Gauss equation. However, it is possible to consider surfaces $S^2 \subseteq \Han$  with finite total extrinsic curvature $\int_S \Vert B^S\Vert^2 d\sigma <\infty$, and this is what Chen Qing and Chen Yi did in \cite{Ch2} and \cite{Che3}.

They proved, for a complete minimal surface $S^2$ (properly) immersed in $\Han$ and such that $\int_S\Vert B^S\Vert^2 d\sigma <\infty$, the following version of the Chern-Osserman Inequality, in terms of the volume growth of the extrinsic balls:
\begin{equation}\label{ChernOssHyp}
\begin{aligned}
 \Sup_{r>0}\frac{\Vol(S^2\cap B^{-1,n}_r)}{\Vol(B^{-1,2}_r)}&< \infty\,\,\,\text{and}\\
 -\chi(S)\leq \frac{1}{4\pi}\int_S\Vert B^S\Vert^2 d\sigma &-\Sup_{r}\frac{\Vol(S^2\cap B^{-1,n}_r)}{\Vol(B^{-1,2}_r)}
\end{aligned}
\end{equation}

The proof of these authors entails elaborate computations which depend on the properties of the hyperbolic functions, far from the complex analysis techniques used in the Euclidean case. 

A natural question which arises in this setting is: do we have an analogous formula when we consider complete minimal surfaces that are properly immersed in a Cartan-Hadamard manifold with sectional curvatures bounded from above by a strictly negative quantity $b <0$? In this paper we provide a (partial) answer to this question. Namely, we have proven that this formula holds for complete minimal surfaces that are properly immersed in an ambient Cartan-Hadamard manifold, with the Hilbert-Schmidt norm of its second fundamental form controlled by $h_b(r)$, the mean curvature (pointed inward) of the geodesic spheres  $S^{b,n-1}_r$ and with finite total extrinsic curvature. We also assume that our ambient Cartan-Hadamard manifold is not too curved (on average) with respect to the Hyperbolic space with constant sectional curvature given by the upper bound $b$. 

To state the first of our main results, it must be remembered (see, for example, \cite{O'N}) that 

$$
h_b(r)=\left\{
\begin{array}{l}
\sqrt{b}\cot\sqrt{b}r\,\,\text{  if }\,\,b>0\\
1/r\,\,\text{  if }\,\, b=0\\
\sqrt{-b}\coth\sqrt{-b}r\,\, \text{  if }\,\, b<0
\end{array}\right.
$$

\noindent We have the following:

\begin{theorem}\label{thmMain1} 
Let $S^{2}$ be a properly immersed minimal surface in
a Cartan-Hadamard manifold $N$, with sectional curvatures bounded from above by a negative quantity $K_N \leq b <0$.

Let us suppose that $\Vert A^S\Vert(q) < h_b(r(q))$ outside a compact set $K \subset S$, where $r(q)=dist_N(o,q)$ denotes the distance of $q \in S$ to a fixed pole $o \in N$ and that

\begin{equation}\label{Hypo1}
\int_{S}\Vert A^{S}\Vert ^{2}d\sigma<+\infty
\end{equation}
\noindent and 
\begin{equation}\label{Hyp2}
\int_{S}(b-K_{N}\vert_S)d\sigma<+\infty
\end{equation}
\noindent where $A^{S}$ denotes the second fundamental form of $S$ in $N$ and $K_{N}\vert_S$ denotes the sectional curvature of $N$ restricted to the tangent plane $T_qS$, for all $q \in S$.

Then:

\begin{enumerate}
\item $\Sup_{t>0}\frac{\operatorname{Vol}(E_{t})}{\operatorname{Vol}(B_{t}^{b,2}
)}<+\infty,$ 
\item $S^2$ has finite topological type,

\item $-\chi(S)\leq\frac{1}{4\pi}
\int_{S}\Vert A^{S}\Vert ^{2}d\sigma-\Sup_{t>0}\frac{\operatorname{Vol}
(E_{t})}{\operatorname{Vol}(B_{t}^{b,2})}+\frac{1}{2\pi}\int_{S}(b-K_{N}\vert_S)d\sigma. \\$
\end{enumerate}
where $E_t=B^N_{t}(o)\cap S$ denotes the $t$-extrinsic ball on surface $S$, centered at $o \in N$ (see definition \ref{ExtBall}), $B^N_t(o)$ is the geodesic $t$-ball centered at the pole $o$ in the ambient space $N$, and $B_{t}^{b,2}$ denotes the geodesic $t$-ball in $\Hatwo$.
\end{theorem}

\begin{remark} The main theorem in \cite{Che3} is a corollary of Theorem \ref{thmMain1}. In fact, 
note that condition (\ref{Hyp2}) is superfluous when the ambient manifold is $\Han$. On the other hand, when the ambient manifold is $\Han$, then condition  (\ref{Hypo1}) implies that $\Vert A^S\Vert(q)$ goes to $0$ as the distance $r(q)$ goes to infinity (see Theorem 2.1 in \cite{O}), so we have  that $\Vert A^S\Vert(q) < h_b(r(q))$ outside a compact set $K \subset S$ and we recover the complete statement of the main theorem in \cite{Che3}.
\end{remark}

\begin{remark}
By applying the Gauss formula, if the surface $S^2$ is minimal, the quantity $b-K_{N}\vert_S$ restricted to $S$ only depends on the points $p \in S$. Hence the assumption $\int
_{S}(b-K_{N}\vert_S)d\sigma<+\infty$ makes sense. We shall denote as $K_N$ the restricted $K_N\vert_S$ when there is no risk of confusion.
\end{remark}

Our proof of Theorem \ref{thmMain1} basically follows the lines of argument used in the proofs given in \cite{Ch2} and \cite{Che3}. A basic fact used in these proofs is the monotonicity property satisfied by the volume growth of the extrinsic balls in minimal surfaces that are properly immersed in the real space forms $\kan$ with $b \leq 0$, namely, that the function $\frac{\operatorname{Vol}(E_{t})}{\operatorname{Vol}(B_{t}^{b,2})}$ is a non-decreasing function of $r$. We have the same monotonicity property when we consider the extrinsic balls on a surface $S$ that is properly immersed in a Cartan-Hadamard manifold $N$ with negative and variable sectional curvature bounded from above by $b<0$. This monotonicity property comes from certain isoperimetric inequalities satisfied by the extrinsic balls in this context which are, in turn, based on the application of a divergence theorem to comparisons of the Laplacian of the extrinsic distance defined on the surface. As we can see in \cite{GW} (see also \cite{MM} and \cite{Pa3}), this comparison arises from the Index lemma, which provides a formula for the Hessian of the distance function in terms of the index form along the normal geodesics to the surface of the Jacobi fields satisfying some given initial conditions. 

Following the break with the framework given by the constant curvature of the ambient space $\Han$ in the works \cite{Ch2} and \cite{Che3}, we have had to overcome several analytical and topological difficulties. 

First, we have extended the Hessian analysis of the extrinsic distance alluded to earlier (which is used in a restricted way in \cite{Ch2} and \cite{Che3} for surfaces in the real space forms $\Han$) to surfaces in Cartan-Hadamard manifolds by using comparison results for the Hessian and the Laplacian of a radial function that can be found in \cite{MP1}, \cite{MP5}, and \cite{HP}. These results are, in turn, based on the Jacobi-Index analysis for the Hessian of the distance function given in \cite{GW}, which we have mentioned previously (see the results in subsection \S3.1). 

Second, and based on this comparison analysis, we have extended the application of the Gauss-Bonnet theorem (which we find in \cite{Che3} restricted to extrinsic balls on surfaces of Hyperbolic space) to the extrinsic balls in minimal surfaces in a Cartan-Hadamard manifold in order to obtain estimates for the Euler characteristic of these extrinsic domains (see the results in subsection \S3.2). 

Third, we present the following estimation of the Euler characteristic of an immersed surface
$$-\chi(S) = \lim_{t \to \infty} (-\chi(E_t))$$
for a suitable exhaustion of $S$ by extrinsic balls $\{E_t\}_{t>0}$ (see Theorem \ref{CorDifeo} in section \S.4). This is a key result which will allow us to argue in a similar way to the line taken in \cite{Ch2} and \cite{Che3}, even though our ambient manifold has no constant curvature. Thanks to the lower bound of the geodesic curvature of extrinsic spheres $\partial E_t$ and to the bound $\Vert A^S\Vert (q) <h_b(r(q))$ outside a compact, it is possible to show that the extrinsic distance to a fixed pole, defined on surface $S$, has no critical points outside a compact. Hence, we can apply classical Morse theory to conclude that, for an exhaustion of $S$ by extrinsic balls $\{E_t\}_{t>0}$, $\chi(E_t)$ is independent of $t$, for a sufficiently large $t$. Therefore $\chi(S)=\lim_{t \to \infty} \chi(E_t)$. When the ambient manifold is the Euclidean or the Hyperbolic space, the bound $\Vert A^S\Vert(q) <h_b(r(q))$ can be omitted because, in this case, the finiteness of the total extrinsic curvature implies that $\Vert A^S\Vert(q)$ goes to $0$ as the extrinsic distance $r(q)$ goes to infinity (for more details, see the proofs of Theorem 2.1 in \cite{O}, concerning cmi surfaces in $\Han$ and Theorem 4.1 in \cite{A1}, about cmi submanifolds in $\erre^n$).

Another appropriate observation at this point is the following: the upper bound $b$ on the sectional curvatures of the ambient manifold $N$ must be strictly negative, because if we use the Euclidean space as a model, the volume of the extrinsic balls $v(t)=\Vol(E_t)$ is not balanced by a function of exponential growth but by the volume function $\Vol(B^{0,2}_t)=\pi t^2$ with slower parabolic growth, and hence the techniques used do not guarantee that $\Sup_{t>0}\frac{\operatorname{Vol}(E_{t})}{\operatorname{Vol}(B_{t}^{0,2}
)}<+\infty$.


To illustrate the meaning of the expression \lq\lq not too curved on average with respect to the hyperbolic space", we will refer to Cartan-Hadamard manifolds, which are asymptotic to Hyperbolic space $\Han$ in a sense that we define below in the following Definition \ref{asymptLoc} (see \cite{Shi}). 

\begin{definition}\label{asymptLoc}
Let us consider $N^n$ a complete non-compact Riemannian manifold with a pole $o \in N$. Then $N$ is {\em asymptotically locally $b$-hyperbolic} of order $\alpha$ (abbreviated as $\alpha$-ALH) if and only if $\vert K_N(x)-b\vert=O(e^{-\alpha r(x)})$, where $K_N(x)$ is the sectional curvature of $N$ at $x \in N$ of the radial planes from the pole $o$ and $r(x)=dist_N(o,x)$ is the distance function from the pole $o\in N$.
\end{definition}

These ambient manifolds satisfy hypothesis (\ref{Hyp2}) of Theorem \ref{thmMain1}, so we have the second of our main results, Theorem \ref{thmMain2}.

\begin{theorem}\label{thmMain2}
Let $S^{2}$ be a properly immersed minimal surface in
a  Cartan-Hadamard manifold $N$ which is asymptotically locally $b$-hyperbolic of order $2$ and with sectional curvatures bounded from above by a negative quantity $K_N \leq b <0$.

Let us suppose that $\Vert A^S\Vert(q) < h_b(r(q))$ outside a compact set $K \subset S$ and that

\begin{equation}
\int_{S}\Vert A^{S}\Vert ^{2}d\sigma<+\infty
\end{equation}

\noindent where $A^{S}$ denotes the second fundamental form of $S$ in $N$.

Then:

\begin{enumerate}
\item $\Sup_{t>0}\frac{\operatorname{Vol}(E_{t})}{\operatorname{Vol}(B_{t}^{b,2}
)}<+\infty,$ 
\item $S^2$ has finite topological type,

\item $-\chi(S)\leq\frac{1}{4\pi}
\int_{S}\Vert A^{S}\Vert ^{2}d\sigma-\Sup_{t>0}\frac{\operatorname{Vol}
(E_{t})}{\operatorname{Vol}(B_{t}^{b,2})}+\frac{1}{2\pi}\int_{S}(b-K_{N})d\sigma. \\$
\end{enumerate}
\end{theorem}

To conclude we have the following generalization of Theorem 3 in \cite{Ch2}.

\begin{theorem}\label{thmMain3}
\label{thmMain3} Let $S^{2}$ be a properly immersed minimal surface in
a Cartan-Hadamard manifold $N$, with sectional curvatures bounded from above by a negative quantity $K_N \leq b <0$. Let us consider an exhaustion of $S$ by a family of nested extrinsic balls $\{E_t=\{x \in S/r(x)\leq t\}\}_{t>0}$, where $r$ is the distance to a fixed pole $o \in S$. Let us suppose that $\lim_{t \to \infty} \frac{\int_{E_t} \cosh r d\sigma}{\cosh^2 t} =\frac{\pi}{-b}$.\\

 (i) Then, $S$ is a minimal cone in $N$ and $\chi(S) =1$.
  
 (ii)  If $N=\Han$, then $S$ is totally geodesic (and we have Theorem 3 in \cite{Ch2}).

\end{theorem}

\subsection{Outline of the paper}
The outline of the paper is as follows. In section \S.2 we present the basic definitions and facts about the extrinsic distance restricted to a submanifold, and about the rotationally symmetric spaces used as a model for comparison purposes. In section \S.3 we present the basic results concerning the Hessian comparison theory of restricted distance function we are going to use, obtaining as a corollary an estimate of the geodesic curvature of the boundary of the extrinsic balls covering the surface and, hence, an estimation of the Euler characteristic of such extrinsic balls. Section \S.4 presents the monotonicity property satisfied by the extrinsic balls and the estimation of the Euler characteristic of the surface in terms of the Euler characteristic of the extrinsic balls. Section \S.5 is devoted to the proof of Theorem \ref{thmMain1}, section \S.6 to the proof of Theorem \ref{thmMain2}, and section \S.7 to the proof of Theorem \ref{thmMain3}.


\section{Preliminaries}

\subsection{Curvature restrictions and extrinsic balls}\label{PrelimOne} We assume throughout the paper that $\varphi: S \longrightarrow N$ is a complete, proper and minimal immersion of a non-compact surface $S$ in a Cartan-Hadamard manifold $N$. Throughout the paper, we identify $\varphi(S)\equiv S$ and $\varphi(x)\equiv x$ for all $ x \in S$. We also assume that the Cartan-Hadamard manifold $N^{n}$ has 
sectional curvatures bounded from above by a negative bound $K_N \leq b <0$. All the
points in these manifolds are poles. Recall that a pole is a point $o$ such
that the exponential map $\exp_{o}\colon T_{o}N^{n}\rightarrow N^{n}$ is a
diffeomorphism. For every $x\in N^{n}\setminus\{o\}$ we define
$r(x)=\operatorname{dist}_{N}(o,x)$, and this distance is realized by the
length of a unique geodesic from $o$ to $x$, which is the \textit{radial
geodesic from $o$}. We also denote by $r$ the restriction $r|_{S}
:S\rightarrow\mathbb{R}_{+}\cup\{0\}$. This restriction is called the
\emph{extrinsic distance function} from $o$ in $S$. The gradients of $r$
in $N$ and $S$ are denoted by $\nabla^{N}r$ and $\nabla^{S}r$, respectively.
Let us remark that $\nabla^{S}r(x)$ is just the tangential component of
$\nabla^{N}r(x)$ in $S$, for all $x\in S$. Then we have the following basic
relation:
\begin{equation}
\nabla^{N}r=\nabla^{S}r+(\nabla^{N}r)^{\bot} \label{radiality}
\end{equation}
where $(\nabla^{N}r)^{\bot}(x)=\nabla^{\bot}r(x)$ is perpendicular to $T_{x}S$
for all $x\in S$.

\begin{definition}
\label{defRadCurv} Let $o$ be a point in a Riemannian manifold $M$ and let
$x\in M\setminus\{o\}$. The sectional curvature $K_{M}(\sigma_{x})$ of the
two-plane $\sigma_{x}\in T_{x}M$ is then called an \textit{$o$-radial
sectional curvature} of $M$ at $x$ iff $\sigma_{x}$ contains the tangent vector
to a minimal geodesic from $o$ to $x$. We also denote these curvatures by
$K_{o,M}(\sigma_{x})$.
\end{definition}

\begin{definition}
\label{ExtBall} Given a connected and complete surface $S$ in a
Cartan-Hadamard manifold $N^{n}$, we denote the \emph{extrinsic metric balls}
of radius $R$ and center $o\in N$ by $E_{R}(o)$. They are defined as the
intersection
\[
E_R=B^N_{R}(o)\cap S=\{x\in S\colon r(x)<R\}
\]
where $B^N_{R}(o)$ denotes the open geodesic ball of radius $R$ centered at the
pole $o$ in $N^{n}$.
\end{definition}

\begin{remark}\label{theRemk0} It should be pointed out that the extrinsic domains $E_{R}(o)$
are precompact sets (because the submanifold $S$ is properly immersed), with
a smooth boundary $\partial E_R=\Gamma_{R}(o)= \{x\in S\colon r(x)=R\}$. The assumption on the smoothness of $\Gamma_{R}(o)$ makes no restriction. Indeed, the distance function $r$ is
smooth in $N^{n}\setminus\{o\}$, since $N^{n}$ is assumed to possess a pole
$o\in N^{n}$. Hence the restriction $r|_{S}$ is smooth in $S$ and consequently
the radii $R$ that produce smooth boundaries $\Gamma_{R}(o)$ are dense in
$\mathbb{R}$ by Sard's theorem and the Regular Level Set Theorem.
\end{remark}

\begin{remark}\label{theRemk1}
When the surface $S$ is totally geodesic in the ambient manifold $N$, the extrinsic $R$-balls become geodesic balls in $S$, $B^S_R$ , and its boundaries are the distance spheres $\partial B^S_R$. On the other hand, when $S$ is a totally geodesic hyperbolic plane in the Hyperbolic space form $\Han$, the extrinsic $R$-ball $E_R$ becomes the geodesic $R$-ball $B^{b,2}_R$ in $\Hatwo$, with boundary $S^{b,1}_R$, the geodesic $R$-sphere in $\Hatwo$.
\end{remark}

For the sake of completeness, we are going to state the co-area formula in these preliminaries. To do so, we shall consider a proper $C^{\infty}$ function $f:M \longrightarrow \erre$ defined on a Riemannian manifold $M$. The set of critical values of $f$ is a null set of $\erre$ and the set of regular values $O$ is an open subset of $\erre$. Then, for $t \in O$, $f^{-1}(t)=\Gamma_t=\{p\in M :f(p)=t\}$ is a compact hypersurface of $M$ and, given $q \in \Gamma_t$, $\nabla^Mf(q)$ is perpendicular to $\Gamma_t$. We define $\Omega_t=\{p \in M: f(p) \leq t\}$ and $v(t)= \Vol(\Omega_t)$. Then

\begin{theoremA}[See \cite{S}, Theorem 5.8]\label{coarea}

Let $M$ be
a Riemannian manifold. Let $f$ be a proper $C^{\infty}$ function defined on
$M$. For an integrable function $u$ on $M$ the following hold:\newline

\begin{enumerate}
\item Let $g_{t}$ be the induced metric on $\Gamma_{t}:=\{p\in M;f(p)=t\}$
from $g$. Then
\[
\int_{M}u\Vert \nabla f\Vert d\nu_{g}=\int_{-\infty}^{+\infty
}dt\int_{\Gamma_{t}}u~d\nu_{g_{t}}
\]

\item The function $t\rightarrow v(t)$ is a $C^{\infty}$
function at regular values $t$ of $f$ such that $V(t)<+\infty$, and
\[
\frac{d}{dt}v(t)=\int_{\Gamma_{t}}\Vert \nabla f\Vert ^{-1}
d\nu_{g_{t}}
\]

\end{enumerate}
\end{theoremA}

\begin{remark} Let us consider an exhaustion of $S$ by a family of nested extrinsic balls $\{E_t\}_{t>0}$, centered at a pole $o \in N$.
To apply the co-area formula in this setting, we consider the surface $S$ as the Riemannian manifold 
and the function $f$ in the above statement is the extrinsic distance from the pole $f=r$. Hence, each extrinsic ball $E_t=\Omega_t$ , the extrinsic spheres are the curves $\partial E_t=\Gamma_t=\{x \in S/r(x)=t\}$, and $v(t)=\Vol(E_t)$ is the volume function.
\end{remark}


\subsection{Warped products and model spaces}\label{subsecWarp}

Warped products are generalized manifolds of revolution. We refer to \cite{O'N} for more information about these spaces. 



\begin{definition}
[See \cite{GW}, \cite{Gri}]\label{defModel} A $w-$model $M_{w}^{m}$ is a
smooth warped product $$M_{w}^{m}=[0,\Lambda[\times_w  \mathbb{S}^{m-1}_{1}$$
\noindent with base $B^{1} = [0,\Lambda[ \,\subset\mathbb{R}$
(where $0 < \Lambda\leq\infty$), fiber $F^{m-1} = \mathbb{S}^{m-1}_{1}$ (i.e.,
the unit $(m-1)$-sphere with standard metric), and warping function
$w\colon[0,\Lambda[ \to\mathbb{R}_{+}\cup\{0\}$, with $w(0) = 0$, $w^{\prime
}(0) = 1$, and $w(r) > 0$ for all $r > 0$. The point $o_{w} = \pi^{-1}(0)$,
where $\pi$ denotes the projection onto $B^{1}$, is called the \emph{center
point} of the model space. If $\Lambda= \infty$, then $o_{w}$ is a pole of
$M_{w}^{m}$.
\end{definition}

\begin{proposition}[See \cite{Gri}, \cite{O'N}]
\label{propSpaceForm} The simply connected space forms $\mathbb{K}^{n}(b)$ of
constant curvature $b$ are $w_b-$models with warping functions
\begin{equation}
w_{b}(r)=
\begin{cases}
\frac{1}{\sqrt{b}}\sin(\sqrt{b}\,r) & \text{if $b>0$}\\
r & \text{if $b=0$}\\
\frac{1}{\sqrt{-b}}\sinh(\sqrt{-b}\,r) & \text{if $b<0$}.
\end{cases}
\end{equation}
Note that for $b>0$ the function $w_{b}(r)$ admits a smooth extension to
$r=\frac{\pi}{\sqrt{b}}$. 

\end{proposition}

\begin{proposition}
[See \cite{GW}, \cite{Gri} and \cite{O'N}]\label{propWarpMean} Let
$M_{w}^{m}$ be a $w-$model with warping function $w(r)$ and center $o_{w}$.
The distance sphere of radius $r$ and center $o_{w}$ in $M_{w}^{m}$ is the
fiber $\pi^{-1}(r)$. This distance sphere has the constant mean curvature
$\eta_{\omega}(r)=\frac{w^{\prime}(r)}{w(r)}$. On the other hand, the
$o_{w}$-radial sectional curvatures of $M_{w}^{m}$ at every $x\in\pi^{-1}(r)$
(for $r>0$) are all identical and determined by
\begin{equation}
K_{o_{w},M_{w}}(\sigma_{x})=-\frac{w^{\prime\prime}(r)}{w(r)}.
\end{equation}
\end{proposition}
\begin{remark}
Note that, for the space forms $\kan$, $\eta_{\omega_{b}}(r)=h_b(r)$.
\end{remark}

\section{Hessian analysis, Gauss-Bonnet Theorem, and estimates for the Euler characteristic of the extrinsic balls}


\subsection{Hessian and Laplacian comparison analysis}\label{subsecLap} 

We now assume that $S^{2}$ is a complete, non-compact,
and properly immersed surface (not necessarily minimal) in a Riemannian manifold $N^{n}$ that possesses a pole $o$. 

The 2nd order analysis of the restricted distance function
$r_{|_{S}}$ is governed
by the Hessian comparison Theorem A in \cite{GW}:

\begin{theoremA}[See \cite{GW}, Theorem A]\label{thmGreW} Let $N=N^{n}$ be a manifold with a
pole $o$, let $M=M_{w}^{m} $ denote a $w-$model with center $o_{w}$, and $m
\leq n$. Suppose that every $o$-radial sectional curvature at $x \in N
\setminus\{o\}$ is bounded from above by the $o_{w}$-radial sectional
curvatures in $M_{w}^{m}$ as follows:
\[
K_{o, N}(\sigma_{x}) \geq\, (\leq)\, -\frac{w^{\prime\prime}(r)}{w(r)}
\]
for every radial two-plane $\sigma_{x} \in T_{x}N$ at distance $r = r(x) =
\operatorname{dist}_{N}(o, x)$ from $o$ in $N$. Then the Hessian of the
distance function in $N$ satisfies
\begin{equation}
\label{eqHess}\begin{aligned} \operatorname{Hess}^{N}(r(x))(X, X) &\leq\,(\geq)\, \operatorname{Hess}^{M}(r(y))(Y, Y)\\ &= \eta_{w}(r)\left(1 - \langle \nabla^{M}r(y), Y \rangle_{M}^{2} \right) \\ &= \eta_{w}(r)\left(1 - \langle \nabla^{N}r(x), X \rangle_{N}^{2} \right) \end{aligned}
\end{equation}
for every unit vector $X$ in $T_{x}N$ and for every unit vector $Y$ in
$T_{y}M$ with $\,r(y) = r(x) = r\,$ and $\, \langle\nabla^{M}r(y), Y
\rangle_{M} = \langle\nabla^{N}r(x), X \rangle_{N}\,$.
\end{theoremA}

\begin{remark}
In \cite[Theorem A, p. 19]{GW}, the Hessian of $r_{M}$ is less than or equal to
the Hessian of $r_{N}$ provided that the radial curvatures of $N$ are bounded
from above by the radial curvatures of $M$ and provided that $\dim M \geq\dim
N$. But $\operatorname{Hess}^{M_{w}}(r(y))(Y, Y)$ \emph{do not} depend
on the dimension $m$, as we can easily see by computing it directly (see
\cite{Pa3}), so the hypothesis on the dimension can be overlooked in the
comparison among the Hessians in this case.
\end{remark}

As a consequence of this result, we have the following Laplacian inequalities (see \cite{MP1},  \cite{Pa3}, or \cite{HP} for detailed developments):

\begin{proposition}\label{corLapComp} 

Let $N^{n}$ be a manifold with a pole $o$, let $M_{w}^{m}$
denote a $w-$model with center $o_{w}$. Let us suppose that every $o$-radial sectional curvature at $x\in N-\{o\}$ is
bounded from above by the $o_{w}$-radial sectional curvatures in $M_{w}^{m}$
as follows:
\begin{equation}
\mathcal{K}(\sigma(x))\,=\,K_{o,N}(\sigma_{x})\leq-\frac{w^{\prime\prime}
(r)}{w(r)} \label{eqKbound}
\end{equation}
for every radial two-plane $\sigma_{x}\in T_{x}N$ at distance
$r=r(x)=\operatorname{dist}_{N}(o,x)$ from $p$ in $N$\medskip

Let $S^2$ be a properly immersed surface in $N$. Let us consider a modified-distance smooth function $f\circ r: S \longrightarrow \erre$. Then:\\

\noindent (A) For such a
smooth function $f(r)$ with $f^{\prime}(r)\leq0\,\,\text{for all}\,\,r$, (respectively $f^{\prime}(r)\geq0\,\,\text{for all}\,\,\,r$), and given
$X\in TqS$ unitary:
\begin{equation}\label{HessFunc1}
\begin{aligned}
\operatorname{Hess}^{S}(f\circ r)(X,X)\,&\leq\, (\geq)  \,(  \,f^{\prime
\prime}(r)-f^{\prime}(r)\eta_{w}(r)\,)  \langle X,\nabla^{N}r \rangle^{2}\\
  +f^{\prime}(r)(  \,\eta_{w}(r)&+\langle \,\nabla^{N}r,\,A^{S}(X,X)\,\rangle )
\end{aligned}
\end{equation}
\newline
(B) Tracing inequality (\ref{HessFunc1})
\begin{equation}\label{LapFunc1}
\begin{aligned}
\Delta^{S}(f\circ r)\,\leq\, (\geq)\, & \,\left(  \,f^{\prime\prime}(r)-f^{\prime
}(r)\eta_{w}(r)\,\right)  \Vert\nabla^{S}r\Vert^{2}\\
&  +mf^{\prime}(r)\left(  \,\eta_{w}(r)+\langle\,\nabla^{N}r,\,H^{S}
\,\rangle\,\right)  
\end{aligned}
\end{equation}
\noindent where $H^S$ denotes the mean curvature vector of $S$ in $N$.
\end{proposition}

Another result we shall use concerning the radial functions defined on the surface is the following:

\begin{proposition}\label{lema 3.1Chen} Let $S^{2}$ be a complete, non-compact,
and properly immersed surface in a Cartan-Hadamard manifold $N^{n}$.
Let us consider $\{E_{t}\}_{t>0}$ an exhaustion of $S$ by extrinsic
balls. Let $f:S\rightarrow\mathbb{R}$ be a positive $C^{\infty}$ function.
Then
\[
\int_{S}e^{-\sqrt{-b}r(x)}~f(x)d\sigma<+\infty \,\,\,\, \text{if and only if}\,\, \int_{0}^{+\infty}e^{-\sqrt
{-b}t}\int_{E_{t}}f(x)~d\sigma~dt<+\infty
\]
\newline and when these integrals converge
\[
\int_{S}e^{-\sqrt{-b}r(x)}~f(x)d\sigma=\int_{0}^{+\infty}e^{-\sqrt{-b}t}
\int_{E_{t}}f(x)~d\sigma~dt
\]

\end{proposition}

\begin{proof}

Given the exhaustion of $S$ by extrinsic balls $\{E_{t}\}_{t>0}$, we apply the co-area formula to obtain, for each $t >0$:

\[
\int_{E_{t}}e^{-\sqrt{-b}r(x)}~f(x)d\sigma=\int_{0}^{t}e^{-\sqrt{-b}s}
\int_{\partial E_{s}}\frac{f(x)}{\Vert \nabla^{S}r\Vert }d\mu ds
\]

\noindent and, on the other hand,
\[
\frac{d}{ds}\int_{E_{s}}f(x)d\sigma=\int_{\Gamma_{s}}\frac{f(x)}
{\Vert \nabla^{S}r\Vert }d\mu~
\]
Hence
\begin{equation}\label{integrals1}
\begin{aligned}
\int_{E_{t}}&e^{-\sqrt{-b}r(x)}~f(x)~d\sigma  =\int_{0}^{t}e^{-\sqrt{-b}
s}\left(  \frac{d}{ds}\int_{E_{s}}f(x)d\sigma\right)  ds\\
&  =e^{-\sqrt{-b}t}\int_{E_{t}}f(x)~d\sigma+\sqrt{-b}\int_{0}^{t}e^{-\sqrt
{-b}s}\int_{E_{s}}f(x)~d\sigma
\end{aligned}
\end{equation}

Taking limits when $t \to \infty$

\begin{equation}\label{integrals}
\begin{aligned}
\int_{S}e^{-\sqrt{-b}r(x)}~f(x)~d\sigma&=\lim_{t\to \infty}\int_{E_t}e^{-\sqrt{-b}r(x)}~f(x)~d\sigma\\=\left(  \underset{t\rightarrow+\infty
}{\lim}e^{-\sqrt{-b}t}\right)  \int_{S}f(x)~d\sigma&+\sqrt{-b}\int_{0}
^{+\infty}e^{-\sqrt{-b}s}\int_{E_{t}}f(x)~d\sigma
\end{aligned}
\end{equation}

\noindent and we have the result because both integrals on the right-hand side of equation (\ref{integrals}) are non-negative.
\end{proof}

\subsection{An application of the Gauss-Bonnet theorem: geodesic curvature of the extrinsic curves on the surface $S$}
 These results have been stated and proven previously in \cite{Ch2} and \cite{Che3}, when the ambient manifold is the hyperbolic space. We extend it here to minimal surfaces in a  Cartan-Hadamard manifold.
 
\begin{proposition}
\label{GeodCurvProp} Let $S^{2}$ be a properly immersed and minimal surface in
a Cartan-Hadamard manifold $N$, with sectional curvatures bounded from above by a negative quantity $K_N \leq b <0$. Let $E_{t}$ be an extrinsic ball in $S$ centered on a pole $o \in N$. The
geodesic curvature of the extrinsic sphere $\partial E_{t}$, denoted as
$k_{g}^{t}$, is bounded from below as follows
\begin{equation}
\begin{aligned}
k_{g}^{t}\geq&  \frac{1}{\Vert \nabla^{S}r\Vert }\left\{  \eta_{\omega_{b}
}+\left\langle A^{S}(e,e),\nabla^{N}r\right\rangle \right\} \\&=\{\eta_{\omega_{b}}(t)-\langle \nabla^{\bot}r,A^{S}(\frac{\nabla^{S}
r}{\Vert\nabla^{S}r\Vert},\frac{\nabla^{S}r}{\Vert\nabla^{S}r\Vert}
)\rangle\}\frac{1}{\Vert\nabla^{S}r\Vert} \label{geodcurvineq}
\end{aligned}
\end{equation}
where $A^S$ denotes the second fundamental form of $S$ in $N$, $e \in TS$ is unitary and tangent to $\Gamma_{t}$ and $\eta
_{\omega_{b}}(t)=h_b(t)$ is the constant mean
curvature of the distance spheres in the hyperbolic spaces $\Han$.
\end{proposition}

\begin{proof}
We apply Proposition \ref{corLapComp} to $f(r)=r$ to conclude that the geodesic curvature $k_{g}^{t}$ satisfies the inequality 
\begin{equation}
\begin{aligned}
k_{g}^{t}  &  =\frac{1}{\Vert \nabla^{S}r\Vert }Hess^{S}
r(e,e)\geq\\
&  \frac{1}{\Vert \nabla^{S}r\Vert }\left\{  -\eta_{\omega_{b}
}\left\langle e,\nabla^{N}r\right\rangle ^{2}+\eta_{\omega_{b}}+\left\langle
A^{S}(e,e),\nabla^{N}r\right\rangle \right\}  =\\
&  \frac{1}{\Vert \nabla^{S}r\Vert }\left\{  \eta_{\omega_{b}
}+\left\langle A^{S}(e,e),\nabla^{N}r\right\rangle \right\}  ,
\end{aligned}
\end{equation}

\noindent where $e \in TS$ is unitary and tangent to $\Gamma_{r}$.

As
\begin{equation}
H^{S}=\frac{1}{2}\left[  A^{S}(e,e)+A^{S}(\frac{\nabla^{S}r}{\Vert
\nabla^{S}r\Vert },\frac{\nabla^{S}r}{\Vert \nabla^{S}r\Vert
})\right]=0  ,
\end{equation}

we obtain:

\begin{equation}
\begin{aligned}
k_{g}^{t}  &  \geq\frac{1}{\Vert \nabla^{S}r\Vert }\left\{
\eta_{\omega_{b}}(t)
-\left\langle A^{S}(\frac{\nabla^{S}r}{\Vert \nabla^{S}r\Vert
},\frac{\nabla^{S}r}{\Vert \nabla^{S}r\Vert }),\nabla^{\perp
}r\right\rangle \right\}  .
\end{aligned}
\end{equation}

\end{proof}

\begin{proposition}
\label{ChaEulerProp} Let $S^{2}$ be a properly immersed and minimal surface in
a Cartan-Hadamard manifold $N$, with sectional curvatures bounded from above by a negative quantity $K_N \leq b <0$. Let $E_{t}$ be a (non-connected) extrinsic ball in $S$ centered on a pole $o \in N$. The
volume $v(t)=\Vol(E_{t})$ satisfies the inequality
\begin{equation}\label{ChaEulerPropIneq}
\begin{split}
2\pi\chi(E_{t})&\geq\eta_{\omega_{b}}(t)v^{\prime}(t)- \int_{\partial E_{t}
}\langle \frac{\nabla^{\bot}r}{\Vert\nabla^{S}r\Vert},A^{S}(\frac{\nabla^{S}r}
{\Vert\nabla^{S}r\Vert},\frac{\nabla^{S}r}{\Vert\nabla^{S}r\Vert}) \rangle d\sigma_t
\\ &+\int_{E_{t}}K_{S}~d\sigma
\end{split}
\end{equation}
where $K_{S}$ denotes the Gaussian curvature of $S$.
\end{proposition}

\begin{proof}
Applying the Gauss-Bonnet theorem

\begin{equation}
\int_{\partial E_{t}}k_{g}^{t}d\mu+\int_{E_{t}}K_{S}d\sigma=2\pi\chi(E_{t}),
\end{equation}

Now, using Proposition \ref{GeodCurvProp}
\begin{equation}
\begin{aligned}
&  2\pi\chi(E_{t})\geq\\
&  \int_{\partial E_{t}}\frac{1}{\Vert \nabla^{S}r\Vert }\left\{
\eta_{\omega_{b}}(t)-\left\langle A^{S}
(\frac{\nabla^{S}r}{\Vert \nabla^{S}r\Vert },\frac{\nabla^{S}
r}{\Vert \nabla^{S}r\Vert }),\nabla^{\perp}r\right\rangle \right\}
d\sigma_t\\&+\int_{E_{t}}K_{S}d\sigma.
\end{aligned}
\end{equation}
\end{proof}

\begin{proposition}
\label{PropIneq} Let $S^{2}$ be a properly immersed and minimal surface in
a Cartan-Hadamard manifold $N$, with sectional curvatures bounded from above by a negative quantity $K_N \leq b <0$. Let $E_{t}$ be an extrinsic ball in $S$ centered on a pole $o \in N$. Then,
given the non-negative real numbers $t>s>0$, we have

\begin{equation}\label{PropIneqIneq}
\begin{split}
\frac{\int_{E_{t}}\cosh\sqrt{-b}rd\sigma}{\cosh^{2}\sqrt{-b}t}-\frac
{\int_{E_{s}}\cosh\sqrt{-b}rd\sigma}{\cosh^{2}\sqrt{-b}s}\\ \geq\int_{E_{t}
-E_{s}}\frac{1+\sinh^{2}\sqrt{-b}r\Vert\nabla^{\bot}r\Vert^{2}}{\cosh^{3}
\sqrt{-b}r}d\sigma
\end{split}
\end{equation}

\end{proposition}

\begin{proof}
As $K_{N}\leq b$ by applying (\ref{LapFunc1}) to the radial function $f(r)=\cosh\sqrt{-b}r$, and as $S$ is minimal, we have,
\begin{equation}\label{compa_cosh}
\Delta^{S}\cosh\sqrt{-b}r\geq-2b\cosh\sqrt{-b}r 
\end{equation}
We integrate inequality (\ref{compa_cosh}) within $E_u$ and then we apply the divergence theorem to obtain
\begin{equation}\label{designablar}
  \sqrt{-b}\sinh\sqrt{-b}u\int_{\Gamma_{u}}\Vert \nabla
^{S}r\Vert d\sigma_u\geq
  -2b\int_{E_{u}}\cosh\sqrt{-b}r~d\sigma
\end{equation}
\newline Therefore
\begin{equation}\label{desig2}
  \int_{E_{u}} \cosh\sqrt{-b}r  ~d\sigma\leq
 \frac{1}{2}\frac{\sinh\sqrt{-b}u}{\sqrt{-b}}\int_{\Gamma_{u}}\Vert
\nabla^{S}r\Vert d\sigma_u
\end{equation}
\newline Deriving and using the inequality above
\begin{equation*}
\begin{aligned}
&  \frac{d}{du}\left(  \frac{\int_{E_{u}} \cosh\sqrt{-b}r d\sigma}{\cosh
^{2}\sqrt{-b}u}\right)  \geq\\
&  \frac{1}{\cosh^{3}\sqrt{-b}u}\left\{  \int_{\Gamma_{u}}\frac{\cosh
^{2}\sqrt{-b}r -\sinh^{2}\sqrt{-b}r\Vert \nabla^{S}r\Vert ^{2}
}{\Vert \nabla^{S}r\Vert }d\sigma_u\right\}  =\\
&  \int_{\Gamma_{u}}\frac{1}{\Vert \nabla^{S}r\Vert }\left\{
\frac{1+\sinh^{2}\sqrt{-b}r\Vert \nabla^{\perp}r\Vert ^{2}
}{\cosh^{3}\sqrt{-b}u}d\sigma_u\right\}
\end{aligned}
\end{equation*}

Now, integrate the inequality above between $s$ and $t$ and apply the co-area formula.
\end{proof}

\section{Extrinsic isoperimetry, volume growth, and topology of surfaces}\label{subsecIsopTop} 

 As mentioned in the Introduction, two key ingredients for our proof of the Chern-Osserman inequality are the following results: an isoperimetric inequality established in \cite{Pa} for the extrinsic balls of minimal submanifolds in Cartan-Hadamard manifolds (and also a monotonicity result which is derived from it and from the co-area formula (see \cite{MP} and \cite{A2})), and a result which relates the Euler characteristic of a surface with the limit value of the Euler characteristic of the sets of an exhaustion by connected extrinsic balls of such a surface.

The first of these results is stated as follows:

\begin{theoremA}\label{isopT}(see \cite{A2}, \cite{MP}, \cite{Pa})
Let $P^m$ be a minimal submanifold properly immersed in a Cartan-Hadamard manifold $N^n$ with sectional curvature $K_N \leq b\leq 0$.
Let $E_r$ be an extrinsic $r$-ball in $P^m$,
 with center at a point $o$ which is also a pole in the ambient space $N$. Then
\begin{equation} \label{isopComp}
\frac{\Vol(\partial E_r)}{\Vol(E_r)} \geq
\frac{\Vol(S^{b,m-1}_r)}{\Vol(B^{b,m}_r)} \,\,\,\,\,\textrm{for all}\,\,\,
r>0  \quad .
\end{equation}
\noindent and 
\begin{equation}\label{isopCompdos}
\frac{\Vol(\partial E_r)}{\Vol(E_r)} \geq
(m-1)h_b(r)\,\,\,\,\,\textrm{for all}\,\,\,
r>0
\end{equation}
Furthermore, the function $f(r)= 
\frac{\Vol(E_r)}{\Vol(B^{b,m}_r)}$ is monotone
non-decreasing in $r$.
    
Moreover, if the equality in inequality (\ref{isopComp}) holds
    for some fixed radius $r_0$ 
then $E_{r_0}$ is a
    minimal cone in the ambient space $N^n$, so if $N^n$ is the
    hyperbolic space $\kan$, $b < 0$, then $P^m$ is totally geodesic in
    $\kan$.
\end{theoremA}
\begin{remark}
In \cite{MaPa} there is a comparison among the lower bounds for the isoperimetric quotient in (\ref{isopComp}) and (\ref{isopCompdos}), depending on the sectional curvature $ b \in \erre$. \end{remark}

A particularization for cmi surfaces in a negatively curved Cartan-Hadamard manifold gives the following monotonicity result:

\begin{corollary}[Minimal Monotonicity]\label{minmon}

Let $S$ be a properly immersed and minimal surface in
a Cartan-Hadamard manifold $N$, with sectional curvatures bounded from above by a negative quantity $K_N \leq b <0$.

Then, the functions
\quad $\frac{v(t)}{\cosh(\sqrt{-b}t)-1}$ \quad and \quad $\frac{v(t)}{e^(\sqrt{-b}t)}$\quad are non-decreasing in $[0, +\infty)$, where $v(t)=\Vol(E_t).\\$
\end{corollary}

On the other hand, we also have the following theorem: as we have mentioned in the Introduction, this is a key result which will allow us to argue as in \cite{Ch2} and \cite{Che3}, applying classical Morse theory to conclude that $\chi(S)=\lim_{t \to \infty} \chi(E_t)$ for an exhaustion of $S$ by extrinsic balls $\{E_t\}_{t>0}$.

Recall that an exhaustion of the surface $S$ by extrinsic balls is a sequence of such subsets, centered at the same point $\{E_t\subseteq S\}_{t>0}$, such that:
\begin{itemize}
\item $E_t \subseteq E_s$ when $s\geq t$
\item $\cup_{t>0} E_t=S$
\end{itemize}

Recall too that the Euler characteristic of a (pre) compact set is finite.

\begin{theorem}\label{CorDifeo}
Let $S$ be an complete minimal surface properly immersed in a Cartan-Hadamard manifold $N$ with sectional curvature bounded from above by a negative quantity $K_N \leq b <0$. Let us suppose that $\int_S\Vert A^S \Vert^2 d\sigma<\infty$ and that $\Vert A^S\Vert(q) \leq h_b(r(q))$ outside a compact set $K \subset S$, where $r(q)=dist_N(o,q)$, the distance to a fixed pole $o \in N$. Then 

(i) $S$ is diffeomorphic to a compact surface $S^*$ punctured at a finite number of points. 
\medskip

(ii) For all sufficiently large $t>R_0>0$, $\chi(S)=\chi(E_t)$ and, hence, given $\{E_t\}_{t>0}$ an exhaustion of $S$ by extrinsic balls centered at the pole $o \in N$, 
$$ -\chi(S)= \lim_{t \to \infty}\inf(- \chi(E_t))<\infty$$

\end{theorem}
\begin{proof} 
 Let us consider $\{E_t\}_{t>0}$ an exhaustion of $S$ by extrinsic balls, centered at the pole $o \in N$. 
We apply Proposition \ref{GeodCurvProp} to the smooth curves $\partial E_t=\Gamma_t$. As 
$$-\Vert A^S\Vert \leq \langle A^S(e,e),\gr^{\bot}r\rangle \leq \Vert A^S\Vert $$

\noindent we have, on the points of the curve $ q \in \Gamma_t$,
\begin{equation}\label{grg0}
\begin{aligned}
\Vert \gr^Sr\Vert(q) \cdot k_g^{\Gamma_t}(q)&\geq h_b(r_p(q))+\langle A^S(e,e),\gr^{\bot}r\rangle(q)\,\\&\geq h_b(r_p(q))-\Vert A^S \Vert(q)
\end{aligned}
\end{equation}
As $\Vert A^S\Vert(q) \leq h_b(r(q)) \,\,\,\forall q \in S \setminus K$, we have, for all the points $q \in \Gamma_t$ and for sufficiently large $t$,

\begin{equation}\label{grg1}
\Vert \gr^Sr\Vert(q) \cdot k_g^{\Gamma_t}(q)\ >0
\end{equation}
Hence, $\Vert \gr^Sr\Vert >0$ in $\Gamma_t$, for all sufficiently large $t$. By fixing a sufficiently large radius $R_0$, we can conclude that the extrinsic distance $r_o$ has no critical points in $S\setminus E_{R_0}$.

The above inequality implies that for this sufficiently large fixed radius $R_0$, there is a diffeomorphism:

$$
\Phi: S\setminus E_{R_0} \to \Gamma_{R_0} \times [0,\infty[
$$

In particular, $S$ has only finitely many ends, each of a finite topological type.

To prove this we apply Theorem 3.1 in \cite{Mi}, concluding that, as the extrinsic annuli $A_{R_0,R}(o)= E_R(o) \setminus E_{R_0}(o)$ contains no critical points of the extrinsic distance function $r_o: S\longrightarrow \erre$ because of inequality (\ref{grg0}), then $E_R(o)$ is diffeomorphic to $E_{R_0}(o)$ for all $R \geq R_0$.

The above diffeomorphism implies that we can construct $S$ from $E_{R_0}$ by attaching annuli and that $\chi(S\setminus E_{t})=0$ when $t\geq R_0$.
Then, for all $t>R_0$,

$$
\chi(S)=\chi(E_{t}\cup (S\setminus E_{t}))=\chi(E_{t}) 
$$

\end{proof}

\section{Proof of Theorem \ref{thmMain1} }\label{subsecMain} 

In this Section we are going to prove our main result, (Theorem \ref{thmMain1}), which generalizes the main theorem in \cite{Che3}. 

Let us consider $\{E_{t}\}_{t>0}$ an exhaustion of $S$ by extrinsic
balls centered at the pole $o \in N$. By adding the quantity $bv(t)$ on both sides of inequality (\ref{ChaEulerPropIneq}), using the Gauss formula to  replace $K_{S}$ by $K_{N}-\frac{1}{2}\Vert A^{S}\Vert ^{2}$ in this same inequality and defining $R(t):= \int_{E_t}\Vert A^S\Vert d\sigma$, we have

\begin{equation}\label{des3}
\begin{aligned}
  \eta_{\omega_{b}}(t)v^{\prime}(t)+b~v(t)\leq
  -\int_{E_{t}}(K_{N}-\frac{1}{2}\Vert A^{S}\Vert^{2}
)d\sigma+\\
  \int_{\partial E_{t}}\frac{1}{\Vert \nabla^{S}r\Vert
}\langle A(\frac{\nabla^{S}r}{\Vert \nabla^{S}r\Vert },\frac{\nabla^{S}
r}{\Vert \nabla^{S}r\Vert }),\nabla^{\bot}r\rangle d\sigma_t+2\pi\chi
(E_{t})\\+\int_{E_{t}}b~d\sigma=
  -\int_{E_{t}}(K_{N}-b)d\sigma&+\frac{1}{2}R(t)\\+\int_{\partial E_{t}
}\frac{1}{\Vert \nabla^{S}r\Vert }\langle A(\frac{\nabla^{S}r}{\Vert
\nabla^{S}r\Vert },\frac{\nabla^{S}r}{\Vert \nabla^{S}r\Vert
}),\nabla^{\bot}r\rangle d\sigma_t&+2\pi\chi(E_{t}). \end{aligned}
\end{equation}

From now on, we denote

\begin{equation}\label{eqdef}
I(t)=\int_{\partial E_{t}}\frac{1}{\Vert\nabla^{S}r\Vert}\left\langle
A^{S}\left(  \frac{\nabla^{S}r}{\Vert\nabla^{S}r\Vert},\frac{\nabla^{S}
r}{\Vert\nabla^{S}r\Vert}\right)  ,\nabla^{\perp}r\right\rangle d\sigma_t,
\end{equation}
 It is straightforward to check that

\begin{equation}\label{des4}
\eta_{\omega_{b}}(t)v^{\prime}(t)+b~v(t)=\sqrt{-b}\frac{\cosh^{2}(\sqrt{-b}
t)}{\sinh(\sqrt{-b}t)}\text{~}\frac{d}{dt}\frac{v(t)}{\cosh(\sqrt{-b}t)}.
\end{equation}
Then, inequality (\ref{des3}) becomes
\begin{equation}
\begin{aligned}
\frac{d}{dt}&\frac{v(t)}{\cosh(\sqrt{-b}t)}    \leq\frac{1}{\sqrt{-b}}
\frac{\sinh(\sqrt{-b}t)}{\cosh^{2}(\sqrt{-b}t)}\left\{  -\int_{E_{t}}
(K_{N}-b)d\sigma+\frac{1}{2}R(t)+\right. \\
&  \left. I(t)+2\pi\chi(E_{t})\right\}  
\end{aligned}
\end{equation}

On the other hand, for all $t >0$ we have:
\begin{equation}\label{sin_exp}
\frac{\sinh(\sqrt{-b}t)}{\cosh^{2}(\sqrt{-b}t)}\leq2e^{-\sqrt{-b}t}
\end{equation}
and hence

\begin{equation}\label{exp}
\begin{aligned}
  \frac{d}{dt}\frac{v(t)}{\cosh(\sqrt{-b}t)}&\leq
  \frac{1}{\sqrt{-b}}\{  2e^{-\sqrt{-b}t}\int_{E_{t}}(-K_{N}
+b)d\sigma+e^{-\sqrt{-b}t}R(t)\\&+
  \frac{\sinh(\sqrt{-b}t)}{\cosh^{2}
(\sqrt{-b}t)}I(t)+4e^{-\sqrt{-b}
t}\pi\chi(E_{t})\}  .
\end{aligned}
\end{equation}

By Theorem \ref{CorDifeo}, for all sufficiently large $t >R_0$, $\chi(E_t)=\chi(S)$.
Now, we integrate both sides of inequality (\ref{exp}) between $0$ and a fixed $t >R_0$, and taking into account that $\frac{v(0)}{\cosh(0)}=0$, the definition of $I(t)$, applying the co-area formula and using the fact that, by Theorem \ref{CorDifeo}, $\chi(E_s) \leq \vert \chi(E_s)\vert  = \vert \chi(S)\vert  <\infty \,\,\,\forall s>R_0$:
\begin{equation}\label{cuentagorda}
\begin{aligned}
  &\frac{v(t)}{\cosh(\sqrt{-b}t)}\leq\frac{1}{\sqrt{-b}}\left\{  2\int_{0}
^{t}e^{-\sqrt{-b}s}\int_{E_{s}}(b-K_{N})d\sigma ds\right.\\&+
  \int_{0}^{t}e^{-\sqrt{-b}s}R(s)ds+\int_{0}^t\frac{\sinh(\sqrt{-b}
s)}{\cosh^{2}(\sqrt{-b}s)}I(s)ds \\&+
   4\pi \int_0^{t}\chi(E_{s})e^{-\sqrt{-b}
s}ds\left. \right\} \\ & \leq \frac{1}{\sqrt{-b}}\left\{  2\int_{0}
^{t}e^{-\sqrt{-b}s}\int_{E_{s}}(b-K_{N})d\sigma ds+\right. 
  \int_{0}^{t}e^{-\sqrt{-b}s}R(s)ds\\&+ \int_{0}^t\frac{\sinh(\sqrt{-b}
s)}{\cosh^{2}(\sqrt{-b}s)}I(s)ds+
  \left.  C(0) 
 \right\} 
\end{aligned}
\end{equation}
where 
\begin{equation*}
\begin{aligned}
0 &< C(0)= 4\pi\int_{0}^{R_0}\chi(E_s)e^{-\sqrt{-b}
s} ds+4\pi \vert \chi(S)\vert \int_{R_0}^{\infty}e^{-\sqrt{-b}
s} ds\\&= 4\pi\int_{0}^{R_0}\chi(E_s)e^{-\sqrt{-b}
s} ds+\frac{4 \pi \vert \chi(S)\vert }{\sqrt{-b}}e^{-\sqrt{-b}R_0}<\infty
\end{aligned}
\end{equation*}

We are going to estimate $\Sup_{t>0} \frac{v(t)}{\cosh(\sqrt{-b}t)}$ using the above inequality. To do so, we proceed as follows.

As $\int_{S}\Vert A^{S}\Vert ^{2}d\sigma<+\infty$, then
$\int_{S}e^{-\sqrt{-b}r}\Vert A^{S}\Vert ^{2}d\sigma<+\infty$.

Then, applying Proposition \ref{lema 3.1Chen} to the non-negative function $f=\Vert A^S\Vert^2$,
using hypothesis (\ref{Hypo1}), we have:
\begin{equation}
\int_{0}^{+\infty}e^{-\sqrt{-b}t}R(t)~dt<+\infty\label{acot1}
\end{equation}

By also applying Proposition \ref{lema 3.1Chen} to the non-negative function $f(x)=b-K_{N}(x)$ defined on $S$, and using hypothesis (\ref{Hyp2}) we know that:

\begin{equation}
\int_{0}^{+\infty}e^{-\sqrt{-b}t}\int_{E_{t}}(b-K_{N})d\sigma dt<+\infty
\label{acot2}
\end{equation}

With these estimates we can conclude, by applying the co-area formula and definition (\ref{eqdef}), that:

\begin{equation}\label{acot3}
\begin{aligned}
&\frac{v(t)}{\cosh(\sqrt{-b}t)}\leq C_1(0)+\frac{1}{\sqrt{-b}}\int_{0}^t
\frac{\sinh(\sqrt{-b}s)}{\cosh^{2}(\sqrt{-b}s)}I(s) ds\\& =C_1(0)+ \frac{1}{\sqrt{-b}}\int_{E_{t}}
\frac{\sinh(\sqrt{-b}r)}{\cosh^{2}(\sqrt{-b}r)}\langle A^{S}(\frac{\nabla^{S}
r}{\Vert \nabla^{S}r\Vert },\frac{\nabla^{S}r}{\Vert
\nabla^{S}r\Vert }),\nabla^{\bot}r\rangle d\sigma.\\
\end{aligned}
\end{equation}

\noindent where $C_1(0)=\frac{1}{\sqrt{-b}}\{C(0)+\int_{0}^{+\infty}e^{-\sqrt{-b}t}\int_{E_{t}}(b-K_{N})d\sigma dt+\int_{0}^{+\infty}e^{-\sqrt{-b}t}R(t)~dt\}$ is a positive and  finite constant.\\

To obtain the result, we need the following:

\begin{lemma}
\label{PropIneqCoro}

There is a constant $C_{2}\geq 0$ satisfying
\begin{equation}\label{acot4}
\begin{aligned}
&  \int_{E_{t}}\frac{\sinh(\sqrt{-b}r)}{\cosh^{2}(\sqrt{-b}
r)}\left\langle A^{S}(\frac{\nabla^{S}r}{\Vert \nabla^{S}r\Vert
},\frac{\nabla^{S}r}{\Vert \nabla^{S}r\Vert }),\nabla^{\bot
}r\right\rangle d\sigma\leq\\
&  C_{2}\sqrt{\frac{v(t)}{\cosh(\sqrt{-b}t)}}
\end{aligned}
\end{equation}
\end{lemma}

\begin{proof}
Let us consider $\{e_1,e_2\}$ an orthonormal basis of $T_pS$, ($p \in S$), being $e_1=\frac{\nabla^S r}{\Vert \nabla^S r\Vert}$. Then
\begin{equation}
\Vert A^S(\frac{\nabla^S r}{\Vert \nabla^S r\Vert},\frac{\nabla^S r}{\Vert \nabla^S r\Vert})\Vert^2
\leq \Vert A^S\Vert^2 
\end{equation}
so 
\begin{equation}\label{ineqbot}
\langle A^{S}(\frac{\nabla^{S}r}{\Vert \nabla^{S}r\Vert
},\frac{\nabla^{S}r}{\Vert \nabla^{S}r\Vert }),\nabla^{\bot
}r\rangle \leq \Vert A^S\Vert ~\Vert\nabla^{\bot} r\Vert
\end{equation}

Applying Cauchy-Schwartz Inequality to the functions $$\frac{\Vert A^S\Vert}{(\cosh (\sqrt{-b} r))^{1/2}} \,\,\,\,\text{and}\,\, \,\,\,\,\frac{\sinh (\sqrt{-b} r) \Vert \nabla^{\bot} r\Vert}{(\cosh (\sqrt{-b} r))^{3/2}},$$ we obtain:
\begin{equation}
\begin{aligned}
&  \int_{E_{t}}\frac{\sinh(\sqrt{-b}r)}{\cosh^{2}(\sqrt{-b}
r)}\left\langle A^{S}(\frac{\nabla^{S}r}{\Vert \nabla^{S}r\Vert
},\frac{\nabla^{S}r}{\Vert \nabla^{S}r\Vert }),\nabla^{\bot
}r\right\rangle d\sigma\leq\\
&  \int_{E_{t}}\sinh(\sqrt{-b}r)\Vert A^{S}\Vert \frac{\Vert
\nabla^{\bot}r\Vert }{\cosh^{2}(\sqrt{-b}r)}d\sigma \leq\\
&  \sqrt{\int_{E_{t}}\frac{\Vert A^{S}\Vert ^{2}d\sigma}
{\cosh(\sqrt{-b}r)}}\sqrt{\int_{E_{t}}\frac{\sinh^{2}(\sqrt{-b}r)\Vert
\nabla^{\bot}r\Vert ^{2}d\sigma}{\cosh^{3}(\sqrt{-b}r)}}.\nonumber
\end{aligned}
\end{equation}

Taking $s=0$ in Proposition \ref{PropIneq} we obtain

\begin{equation*}\label{ineq_s=0}
\begin{aligned}
 \int_{E_{t}}\frac{1+\sinh^{2}(\sqrt{-b}r)\Vert \nabla^{\bot
}r\Vert ^{2}}{\cosh^{3}(\sqrt{-b}r)}d\sigma &\leq\,\frac{\int_{E_{t}}\cosh(\sqrt{-b}r)d\sigma}{\cosh^{2}(\sqrt{-b}t)}
\end{aligned}
\end{equation*}

As, on the other hand, $\cosh(\sqrt{-b}r)$ is non-decreasing, then

\begin{equation*}
\frac{\int_{E_{t}}\cosh(\sqrt{-b}r)d\sigma}{\cosh^{2}(\sqrt{-b}t)}\leq
\frac{\cosh(\sqrt{-b}t)v(t)}{\cosh^{2}(\sqrt{-b}t)}=\frac{v(t)}{\cosh
(\sqrt{-b}t)}
\end{equation*}

\noindent Hence
\begin{equation*}
\int_{E_{t}}\frac{\sinh^{2}(\sqrt{-b}r)\Vert \nabla^{\bot}r\Vert
^{2}}{\cosh^{3}(\sqrt{-b}r)}d\sigma\leq\frac{v(t)}{\cosh(\sqrt{-b}t)}
\end{equation*}
and therefore:
\begin{equation*}
\begin{aligned}
&  \int_{E_{t}^{{}}}\frac{\sinh(\sqrt{-b}r)}{\cosh^{2}(\sqrt{-b}r)}
\langle A^{S}(\frac{\nabla^{S}r}{\Vert \nabla^{S}r\Vert },\frac{\nabla
^{S}r}{\Vert \nabla^{S}r\Vert }),\nabla^{\bot}r\rangle\leq\\
&  \sqrt{\int_{E_{t}}\frac{\Vert A^{S}\Vert ^{2}}{\cosh(\sqrt
{-b}r)}}\sqrt{\frac{v(t)}{\cosh(\sqrt{-b}t)}}
\end{aligned}
\end{equation*}

As $\frac{1}{\cosh\sqrt{-b}t}\leq 2e^{-\sqrt{-b}t}\,\,\forall t>0$, we have

\begin{equation*}
0\leq \sqrt{\int_{E_{t}}\frac{\Vert A^{S}\Vert ^{2}d\sigma}{\cosh
(\sqrt{-b}r)}}\leq\sqrt{\int_{S}2e^{-\sqrt{-b}r}\Vert A^{S}\Vert
^{2}d\sigma}=C_2 <\infty
\end{equation*}
because $\int_{S}e^{-\sqrt{-b}r}\Vert A^{S}\Vert
^{2}d\sigma < \infty$ as we have seen before.
\end{proof}

Returning to (\ref{acot3}), and using Lemma \ref{PropIneqCoro}, we have 

\begin{equation*}
\frac{v(t)}{\cosh(\sqrt{-b}t)}\leq C_{1}(0)+C_{2}\sqrt{\frac{v(t)}{\cosh
(\sqrt{-b}t)}}.
\end{equation*}

By putting $h(t)=\sqrt{\frac{v(t)}{\cosh(\sqrt{-b}t)}\text{ }}$ the inequality above becomes:

\begin{equation*}
h^{2}(t)-C_{2}h(t)-C_{1}(0)\leq0
\end{equation*}
\noindent and hence the values of $h(t)$ lie between the zeroes of the function $
f(x)=x^{2}-C_{2}x-C_{1}(0)$, which are real and distinct numbers (because $C_{1}(0)>0$ and $C_2 \geq 0$ and it is not possible that $C_1(0)=C_2=0$). Hence, $h(t)$ (and also $h^2(t)$) are bounded.

We have proven that $\frac{v(t)}{\cosh(\sqrt{-b}t)}<\infty$ and therefore, 
$\frac{v(t)}{\cosh(\sqrt{-b}t)-1} <\infty$, so assertion (1) of the Theorem is proven.\\

To prove assertion (2), we remember equation (\ref{eqdef}) so
that inequality (\ref{des3}) becomes

\begin{equation}
-2\pi\chi(E_{t})\leq-\int_{E_{t}}(K_{N}-b)d\sigma+\frac{1}{2}R(t)+I(t)-\eta
_{\omega_{b}}(t)v^{\prime}(t)-b~v(t) \label{desi_inicio_th2}
\end{equation}

We now need the following

\begin{lemma}\label{lema acotacion integral v} 
$\int_{0}^{t}\cosh(\sqrt{-b}s)~v^{\prime
}(s)ds\geq\frac{\cosh(\sqrt{-b}t)+1}{2}v(t)$
\end{lemma}

\begin{proof}
As $\frac{v(t)}{\cosh(\sqrt{-b}t)-1}$ is non-decreasing, we know that 
\begin{equation}\label{ineq1}
\left(  \cosh(\sqrt{-b}t)-1\right)  v^{\prime}(t)\geq v(t)\sqrt{-b}\sinh
(\sqrt{-b}t)
\end{equation}

Hence, integrating both sides of the inequality above:
\begin{equation*}
\begin{aligned}
&  \int_{0}^{t}\cosh(\sqrt{-b}s)~v^{\prime}(s)ds=\\
&  v(t)\cosh(\sqrt{-b}t)-\sqrt{-b}\int_{0}^{t}v(s)\sinh(\sqrt{-b}s)ds\geq\\
&  v(t)\cosh(\sqrt{-b}t)-\int_{0}^{t}(\cosh(\sqrt{-b}s)-1)v^{\prime}(s)ds=\\
&  v(t)(\cosh(\sqrt{-b}t)+1)-\int_{0}^{t}\cosh(\sqrt{-b}s)~v^{\prime}(s)ds.
\end{aligned}
\end{equation*}
\end{proof}

Again, using the definition of $I(t)$, inequality (\ref{ineqbot}), and the arithmetic-geometric mean inequality $xy\leq\frac{x^{2}+y^{2}}{2}$, we have 
\begin{equation}\label{ineqdeI}
\begin{aligned}
  I(t)&\leq\int_{\partial E_{t}}\Vert A^{S}\Vert \frac{\Vert
\nabla^{\perp}r\Vert }{\Vert\nabla^{S}r\Vert}d\sigma_t\\&=\int_{\partial E_{t}
}\frac{\Vert A^{S}\Vert }{\sqrt{\eta_{\omega_{b}}(t)}\sqrt
{\Vert\nabla^{S}r\Vert}}\frac{\sqrt{\eta_{\omega_{b}}(t)}\Vert \nabla^{\perp}r\Vert }
{\sqrt{\eta_{\omega_{b}}(t)}\sqrt{\Vert\nabla^{S}r\Vert}}d\sigma_t\leq\\
&  \frac{1}{2}\int_{\partial E_{t}}\left(  \frac{\Vert A^{S}\Vert
^{2}}{\eta_{\omega_{b}(t)}\Vert\nabla^{S}r\Vert}+\frac{\eta_{\omega_{b}}(t)\Vert
\nabla^{\perp}r\Vert ^{2}}{\Vert\nabla^{S}r\Vert}\right)  d\sigma_t\leq\\&
\frac{1}{\eta_{\omega_{b}}(t)}\int_{\partial E_{t}}\frac{\Vert
A^{S}\Vert ^{2}}{\Vert\nabla^{S}r\Vert}+\eta_{\omega_{b}}(t)\int
_{\partial E_{t}}\frac{\Vert \nabla^{\perp}r\Vert ^{2}}{\Vert
\nabla^{S}r\Vert}d\sigma_t. 
\end{aligned}
\end{equation}
But, by applying the co-area formula,
\[
\frac{1}{\eta_{\omega_{b}}(t)}R^{\prime}(t)=\frac{1}{\eta_{\omega_{b}}(t)}
\int_{\partial E_{t}}\frac{\Vert A^{S}\Vert ^{2}}{\Vert\nabla
^{S}r\Vert}d\sigma_t,
\]
so we have
\begin{equation}\label{ineqdeII}
I(t) \leq \frac{R'(t)}{\eta_{\omega_{b}}(t)}+ \eta_{\omega_{b}}(t)\int
_{\partial E_{t}}\frac{\Vert \nabla^{\perp}r\Vert ^{2}}{\Vert
\nabla^{S}r\Vert}d\sigma_t.
\end{equation}
 On the other hand, by using the co-area formula, inequality (\ref{desig2}), and Lemma \ref{lema acotacion integral v} we obtain:
\begin{equation}\label{integr_grad}
\begin{aligned}
&  \eta_{\omega_{b}}(t)\int_{\partial E_{t}}\frac{\Vert \nabla^{\perp
}r\Vert ^{2}}{\Vert\nabla^{S}r\Vert}d\mu=\eta_{\omega_{b}}
(t)\int_{\partial E_{t}}\frac{1-\Vert\nabla^{S}r\Vert^{2}}{\Vert\nabla
^{S}r\Vert}d\mu\\ \leq
&  \eta_{\omega_{b}}(t)v^{\prime}(t)-\eta_{\omega_{b}}(t)\int_{\partial E_{t}
}\Vert\nabla^{S}r\Vert d\sigma\\ \leq
&  \eta_{\omega_{b}}(t)v^{\prime}(t)-\frac{2\eta_{\omega_{b}}(t)\sqrt{-b}
}{\sinh(\sqrt{-b}t)}\int_{E_t}\cosh(\sqrt{-b}r)d\sigma\\=
&  \eta_{\omega_{b}}(t)v^{\prime}(t)-\frac{2\eta_{\omega_{b}}(t)\sqrt{-b}
}{\sinh(\sqrt{-b}t)}\int_{0}^{t}\cosh(\sqrt{-b}s)~v^{\prime}(s)ds\\ \leq
&  \eta_{\omega_{b}}(t)v^{\prime}(t)-\frac{v(t)\eta_{\omega_{b}}(t)\sqrt{-b}
}{\sinh\sqrt{-b}t}(\cosh\sqrt{-b}t~+1)\\
&= \eta_{\omega_{b}}(t)v^{\prime}(t)-\eta_{\omega_{b}}(t)^{2}v(t)-\frac
{\sqrt{-b}\eta_{\omega_{b}}(t)v(t)}{\sinh\sqrt{-b}t}
\end{aligned}
\end{equation}

Finally, from (\ref{ineqdeII}) y (\ref{integr_grad}) we obtain:
\begin{equation} \label{des_i}
I(t)\leq\frac{1}{\eta_{\omega_{b}}(t)}R^{\prime}(t)+\eta_{\omega_{b}
}(t)v^{\prime}(t)-\eta_{\omega_{b}}(t)^{2}v(t)-\frac{\sqrt{-b}\eta_{\omega
_{b}}(t)v(t)}{\sinh\sqrt{-b}t}. 
\end{equation}

Now considering (\ref{desi_inicio_th2}), and applying (\ref{des_i}):
\begin{equation} \label{char}
\begin{aligned}
&  -2\pi\chi(E_{t})\leq\int_{E_{t}}(b-K_{N})d\sigma+\frac{1}{2}R(t)+\frac
{1}{\eta_{\omega_{b}}(t)}R^{\prime}(t)\\ +
&  \eta_{\omega_{b}}(t)v^{\prime}(t)-\eta_{\omega_{b}}(t)^{2}v(t)-\left(
\eta_{\omega_{b}}(t)v^{\prime}(t)+b~v(t)\right)  -\frac{\sqrt{-b}\eta
_{\omega_{b}}(t)v(t)}{\sinh\sqrt{-b}t}\\ \leq
&  \int_{E_{t}}(b-K_{N})d\sigma+\frac{1}{2}R(t)+\frac{1}{\eta_{\omega_{b}}
(t)}R^{\prime}(t)\\&+v(t)(-b-\eta_{\omega_{b}}(t)^{2})-\frac{\sqrt{-b}
\eta_{\omega_{b}}(t)v(t)}{\sinh\sqrt{-b}t}
\end{aligned}
\end{equation}

It is straightforward to see, taking into account that $\operatorname{Vol}(B_{t}^{b,2})=\frac{-2\pi}{b}(\cosh
\sqrt{-b}t-1)$,
\begin{equation}
v(t)(-b-\eta_{\omega_{b}}(t)^{2})-\frac{\sqrt{-b}
\eta_{\omega_{b}}(t)v(t)}{\sinh\sqrt{-b}t}=\frac{b v(t)}{\cosh\sqrt{-b}t-1}=\frac{-2\pi v(t)}{\operatorname{Vol}(B_{t}^{b,2})}
\end{equation}

and hence

\begin{equation}\label{ineq5}
\begin{aligned}
-2\pi\chi(E_{t})&\leq\int_{E_{t}}(b-K_{N})d\sigma+\frac{1}{2}R(t)+\frac
{1}{\eta_{\omega_{b}}(t)}R^{\prime}(t)-\frac{2\pi v(t)}{\operatorname{Vol}(B_{t}^{b,2})}
\end{aligned}
\end{equation}

As we define $R(t)=\int_{E_t}\Vert A^{S}\Vert ^{2}d\sigma$, then $\int_{S}\Vert A^{S}\Vert ^{2}d\sigma=\lim_{t \to \infty} R(t)=\int_{0}^{+\infty
}R^{\prime}(t)dt<+\infty$. Therefore, there is a monotone increasing (sub)sequence $\{t_i\}_{i=1}^\infty$ tending to infinity (namely, $t_i \to \infty$ when $i \to \infty$), such that $R'(t_i)\rightarrow 0$ when $i\to \infty$, and hence 
\[
\underset{i\rightarrow+\infty}{\lim}\frac{1}{\eta_{\omega_{b}}(t_{i}
)}R^{\prime}(t_{i})=\frac{0}{\sqrt{-b}}=0.
\]

Let us consider the exhaustion of $S$ by these extrinsic balls, namely, $\{E_{t_i}\}_{i=1}^\infty$. Since $\{E_{t_i}\}_{i=1}^\infty$ is a family of  precompact open sets exhausting $S$, then the sequence $$\{\inf(\{-\chi(E_{r_k})\}_{k=i}^\infty\}_{i=1}^\infty$$ is monotone non-decreasing. Then
we have, by replacing $t$ for $t_i$ and taking limits when $i \to \infty$ in inequality (\ref{ineq5}), that
\begin{equation*}
\begin{aligned}
&\lim_{i\rightarrow \infty}\inf(\{-\chi(E_{r_k})\}_{k=i}^\infty)\\ &\leq \int_{S}(b-K_{N})d\sigma+\frac{1}{2}\int_{S}\Vert
A^{S}\Vert ^{2}d\sigma-2\pi\Sup_{t>0}\frac{v(t)}{\operatorname{Vol}(B_{t}^{b,2})}<\infty
\end{aligned}
\end{equation*}
and hence, by applying Theorem \ref{CorDifeo}, $S^2$ has finite topology and
\begin{equation}
-2\pi\chi(S)\leq\int_{S}(b-K_{N})d\sigma+\frac{1}{2}\int_{S}\Vert
A^{S}\Vert ^{2}d\sigma-2\pi\Sup_{t>0}\frac{ v(t)}{\operatorname{Vol}(B_{t}^{b,2})}
\end{equation}

\section{Proof of Theorem \ref{thmMain2}}
We are going to apply Theorem \ref{thmMain1}, and to do so it is enough to check that hypothesis (\ref{Hyp2}) in Theorem \ref{thmMain1}, i.e.,  inequality 
$$\int_S (b - K_N) d\sigma < \infty$$
\noindent is satisfied in our setting.
By Definition \ref{asymptLoc}, we have that $\vert K_N\vert_S -b\vert \leq K e^{-2\sqrt{-b} r(x)}$, for all $x \in S -E_M(o)$, $E_M(o)$ being an extrinsic ball centered at one pole $o \in N$. Hence, if we consider $\{E_{t}\}_{t>0}$ an exhaustion of $S$ by extrinsic
balls centered at the pole $o \in N$, we have,

\begin{equation}\label{eq1}
\begin{aligned}
\int_S (b - K_N\vert_S) d\sigma &\leq \int_S \vert b - K_N\vert_S \vert d\sigma= \int_{E_M(o)} \vert b - K_N\vert_S \vert d\sigma\\&+ \int_{S-E_{M}(o)} \vert b - K_N\vert_S \vert d\sigma\\ &\leq C_1+K\int_{S-E_{M}(o)} e^{-2\sqrt{-b} r} d\sigma\\&  \leq C_1+K\int_{S} e^{-2\sqrt{-b} r} d\sigma
\end{aligned}
\end{equation}

and, applying the co-area formula as in (\ref{integrals1}) and (\ref{integrals}), we obtain 


\begin{equation}\label{eq3}
\begin{aligned}
\int_S (b - K_N\vert_S) d\sigma &\leq \int_S \vert b - K_N\vert_S \vert d\sigma\\& \leq  C_1+K\int_{S} e^{-2\sqrt{-b} r} d\sigma = C_1 + K \lim_{t \to \infty} v(t) e^{-2\sqrt{-b} t}\\&+ 2K \sqrt{-b}\lim_{t \to \infty}\int_0^t v(s) e^{-2\sqrt{-b} s} ds
\end{aligned}
\end{equation}

To prove the theorem, we must check that $\lim_{t \to \infty} v(t) e^{-2\sqrt{-b} t} < \infty$ and that \newline $\int_0^\infty v(s) e^{-2\sqrt{-b} s} ds <\infty$. To do so, let us consider the non-decreasing function $f(t)=\frac{v(t)}{e^{\sqrt{-b} t}}$ (see Corollary \ref{minmon}). We shall see that $f(t)$ is bounded, that is, that $\lim_{t \to \infty} f(t) <\infty$.

Taking into account the fact that $\eta_{w_b}(t)=\sqrt{-b}\coth(\sqrt{-b} t) \geq \sqrt{-b}\quad \forall t >0$, we obtain
\begin{equation}\label{eq4}
\sqrt{-b} v'(t) + b v(t) \leq \eta_{w_b}(t) v'(t) + b v(t)\quad \forall t >0
\end{equation}

On the other hand,
\begin{equation}\label{eq5}
\sqrt{-b} v'(t) + b v(t)= \sqrt{-b}e^{\sqrt{-b} t} f'(t)
\end{equation}

so, using inequality (\ref{des3}) in the proof of Theorem \ref{thmMain1},

\begin{equation}\label{eq7}
\begin{aligned}
 f'(t) &\leq \frac{1}{\sqrt{-b}}e^{-\sqrt{-b} t} \{ \int_{E_{t}}(b-K_{N})d\sigma+\frac{1}{2}R(t)+I(t)+2\pi\chi(E_{t})\}\\& \leq \frac{1}{\sqrt{-b}}e^{-\sqrt{-b} t} \{ \int_{E_{t}}\vert b-K_{N}\vert d\sigma+\frac{1}{2}R(t)+I(t)+2\pi\chi(E_{t})\} 
 \end{aligned}
 \end{equation}
 
 Now, we integrate both sides of inequality (\ref{eq7}) between $0$ and $t > R_0$ as in the proof of Theorem \ref{thmMain1}. Then:
\begin{equation}\label{eq8}
\begin{aligned}
  f(t)&\leq\frac{1}{\sqrt{-b}}\{  \int_{0}
^{t}e^{-\sqrt{-b}s}\int_{E_{s}}\vert b-K_{N}\vert d\sigma ds\\&+
  \int_{0}^{t}e^{-\sqrt{-b}s}R(s)ds+\int_{0}^t e^{-\sqrt{-b}s} I(s)ds \\&+
  C_2(0) \} 
\end{aligned}
\end{equation}
where, as in the proof of Theorem \ref{thmMain1},
\begin{equation*}
\begin{aligned}
0 &< C_2(0)= 4\pi\int_{0}^{R_0}\chi(E_s)e^{-\sqrt{-b}
s} ds+4\pi \vert \chi(S)\vert \int_{R_0}^{\infty}e^{-\sqrt{-b}
s} ds\\&= 4\pi\int_{0}^{R_0}\chi(E_s)e^{-\sqrt{-b}
s} ds+\frac{4 \pi \vert \chi(S)\vert }{\sqrt{-b}}e^{-\sqrt{-b}R_0}<\infty
\end{aligned}
\end{equation*}

With the same arguments as in the proof of Theorem \ref{thmMain1} and using hypothesis (\ref{Hypo1}), we have 

\begin{equation}\label{eq9}
\begin{aligned}
f(t)&\leq\frac{1}{\sqrt{-b}}\{  \int_{0}
^{t}e^{-\sqrt{-b}s}\int_{E_{s}}\vert b-K_{N}\vert d\sigma ds\\&+\int_{0}^t e^{-\sqrt{-b}s} I(s)ds +
  C_3 \} 
 \end{aligned}
 \end{equation}
 
where $0< C_3 =C_2(0) + \int_S e^{-\sqrt{-b}r}\Vert A^S \Vert ^2d\sigma < \infty$
\medskip
 
Now, we are going to prove the following Lemma:
\medskip

\begin{lemma}
\label{Lemma2}

There is a constant $C_{4} \, >\, 0$ satisfying
\begin{equation}\label{eq10}
\begin{aligned}
\int_{0}^t e^{-\sqrt{-b}s} I(s)ds \leq C_4 \sqrt{f(t)} \quad \forall t >0
\end{aligned}
\end{equation}
\end{lemma}

\begin{proof}

We argue as in Lemma \ref{PropIneqCoro}: by applying Cauchy-Schwartz inequality and the co-area formula, and using inequality (\ref{ineqbot}), we obtain
\begin{equation}\label{eq11}
\begin{aligned}
&  \int_0^t e^{-\sqrt{-b} s} I(s) ds =\int_0^t e^{-\sqrt{-b} s}\int_{\partial D_s}\left\langle A^{S}(\frac{\nabla^{S}r}{\Vert \nabla^{S}r\Vert
},\frac{\nabla^{S}r}{\Vert \nabla^{S}r\Vert }),\nabla^{\bot
}r\right\rangle d\sigma_s ds \\ & \leq
  \int_0^t e^{-\sqrt{-b} s}\int_{\partial D_s}\Vert A^{S}\Vert \frac{\Vert
\nabla^{\bot}r\Vert }{\Vert \nabla^S r\Vert}d\sigma_s ds \leq
  \int_{E_{t}} \frac{\Vert A^{S}\Vert
\Vert \nabla^{\bot}r\Vert d\sigma}
{\sqrt{e^{\sqrt{-b}r}}  \sqrt{e^{\sqrt{-b}r}} } \\& \leq \sqrt{\int_{E_t}\frac{\Vert A^{S}\Vert^2
 d\sigma} {e^{\sqrt{-b}r} }} \sqrt{\int_{E_t}\frac{\Vert
\nabla^{\bot}r\Vert^2
 d\sigma}
{e^{\sqrt{-b}r} }} \leq C_4 \sqrt{\int_{E_t}\frac{\Vert
\nabla^{\bot}r\Vert^2
 d\sigma}
{e^{\sqrt{-b}r} }} 
\end{aligned}
\end{equation}
because $0<\int_{E_t}\frac{\Vert A^{S}\Vert^2
 d\sigma} {e^{\sqrt{-b}r} }= C_4 < \infty$

To conclude the proof of the Lemma, we are going to see that, for all $t >0$,
\begin{equation}\label{eq12}
\int_{E_t}\frac{\Vert
\nabla^{\bot}r\Vert^2
 d\sigma}
{e^{\sqrt{-b}r} } \leq \frac{v(t)}{e^{\sqrt{-b}t}}
\end{equation}

By inequality (\ref{compa_cosh}), we have, for all $r>0$

\begin{equation}\label{equ13}
\Delta^{S}\cosh\sqrt{-b}r\geq-2b\cosh\sqrt{-b}r\geq -be^{\sqrt{-b}r} 
\end{equation}

Integrating two sides of (\ref{equ13}) and applying Divergence theorem, we have
\begin{equation}\label{equ14}
\frac{\sinh{\sqrt{-b}t}}{\sqrt{-b}}\int_{\partial E_t} \Vert \nabla^S r\Vert d\sigma_t \geq \int_{E_t} e^{\sqrt{-b}r}d\sigma
\end{equation}

Deriving the function $\frac{\int_{E_u} e^{\sqrt{-b}r}d\sigma}{e^{2\sqrt{-b}u}}$ and using inequality (\ref{equ14}):

\begin{equation}\label{equ15}
\frac{d}{du}\frac{\int_{E_u} e^{\sqrt{-b}r}d\sigma}{e^{2\sqrt{-b}u}} \geq \int_{\partial E_u} e^{-\sqrt{-b}r}\frac{\Vert\nabla^{\bot}r\Vert^2}{\Vert \nabla^S r\Vert} d\sigma_u
\end{equation}

So, by integrating both sides of (\ref{equ15}) between $0$ and $t$ and using the co-area formula, and the fact that $e^{\sqrt{-b}r}$ is non-decreasing:

\begin{equation}
\begin{aligned}
\int_{ E_t} e^{-\sqrt{-b}r}\Vert \nabla^{\bot} r \Vert^2 d\sigma 
&\leq \frac{\int_{E_t} e^{\sqrt{-b}r}d\sigma}{e^{2\sqrt{-b}t}}
\leq \frac{v(t)}{e^{\sqrt{-b}t}}
\end{aligned}
\end{equation}

Then, there exists $C_4 \geq 0$ such that

\begin{equation}
\int_0^t e^{-\sqrt{-b}s} I(s)ds \leq C_4\sqrt{\int_{ E_t} e^{-\sqrt{-b}r}\Vert\nabla^{\bot}r\Vert^2 d\sigma} \leq C_4 \sqrt{\frac{v(t)}{e^{\sqrt{-b}t}}}
\end{equation}
\end{proof}

Now, using inequality (\ref{eq9}) and Lemma \ref{Lemma2} we have
\begin{equation}\label{eq10}
f(t)\leq\frac{1}{\sqrt{-b}}\{  \int_{0}
^{t}e^{-\sqrt{-b}s}\int_{E_{s}}\vert b-K_{N}\vert d\sigma ds+C_4 \sqrt{f(t)} +
  C_3 \} 
 \end{equation}\\

 We are now going to see that
 \begin{equation}
  \int_{0}
^{t}e^{-\sqrt{-b}s}\int_{E_{s}}\vert b-K_{N}\vert d\sigma ds \leq C_5+K \int_{0}^{t}e^{-\sqrt{-b}s}\int_{E_{s}}e^{-\sqrt{-b} r} d\sigma ds
\end{equation}

As $\vert b-K_{N}(x)\vert=O(e^{-2\sqrt{-b} r(x)})$, namely, there exists $M >0$ and $K >0$ such that $\vert b-K_{N}(x)\vert \leq Ke^{-2\sqrt{-b} r(x)}\leq Ke^{-\sqrt{-b} r(x)}$ for all $x\in S-E_M(o)$, then
\begin{equation}
\begin{aligned} 
\int_{0}^{t}&e^{-\sqrt{-b}s}\int_{E_{s}}\vert b-K_{N}\vert d\sigma ds  \leq \int_{0}^{M}e^{-\sqrt{-b}s}\int_{E_{s}}\vert b-K_{N}\vert d\sigma ds\\ & + \int_{M}^{t}e^{-\sqrt{-b}s}\int_{E_{s}}\vert b-K_{N}\vert d\sigma ds \\ & \leq C_5+ \int_{M}^{t}e^{-\sqrt{-b}s}\{\int_{E_{s}-E_{M}}\vert b-K_{N}\vert d\sigma +\int_{E_{M}}\vert b-K_{N}\vert d\sigma \}ds\\  & \leq C_5+K\int_{0}^{t}e^{-\sqrt{-b}s}\int_{E_{s}}e^{-\sqrt{-b} r} d\sigma ds
\end{aligned}
\end{equation}
 
 Now, using equality (\ref{integrals1}) in Proposition \ref{lema 3.1Chen}, and from the fact that given a fixed $t>0$, $e^{-\sqrt{-b}t} \leq e^{-\sqrt{-b}r}$ for all $r \leq t$, we have
 
 \begin{equation}
 \begin{aligned}
 \sqrt{-b} &\int_{0}^{t}e^{-\sqrt{-b}s}\int_{E_{s}}e^{-\sqrt{-b} r} d\sigma ds= \int_{E_{t}}e^{-2\sqrt{-b}r}d\sigma\\&-e^{-\sqrt{-b} t}\int_{E_{t}}e^{-\sqrt{-b} r} d\sigma  \leq \int_{E_{t}}e^{-2\sqrt{-b}r}\\&-e^{-\sqrt{-b} t}\frac{v(t)}{e^{\sqrt{-b} t}}=2\sqrt{-b}\int_0^t v(s) e^{-2 \sqrt{-b} s} ds\\& =2\sqrt{-b}\int_0^t f(s) e^{- \sqrt{-b} s} ds
 \end{aligned}
 \end{equation}
  and hence, from inequality (\ref{eq10}) and with $\bar C_1:=K > 0$, $\bar C_2:= C_5+C_3 > 0$ and $\bar C_3:=C_4 > 0$
  
  \begin{equation}\label{eq11}
f(t)\leq\frac{1}{\sqrt{-b}}\{ 2\bar C_1\int_0^t f(s) e^{- \sqrt{-b} s} ds+\bar C_2+\bar C_3 \sqrt{f(t)} \} 
 \end{equation} 
 
 On the other hand, $f(t)=\frac{v(t)}{e^{\sqrt{-b} t}} \geq 0$ for all $t >0$ and, as $S$ is minimal, using inequality (\ref{isopCompdos}) in Theorem \ref{isopT}, $f'(t) \geq 0$ for all $t >0$. Moreover, we can assume that there exists $t_0 >0$ such that $f(t) \geq 1$ for all $t \geq t_0$ (in contrast, $f(t) \leq 1 \,\,\forall t >0$ and the theorem is proven using inequality (\ref{eq3})). Hence, $f(t) \geq \sqrt{f(t)}$ for all $t \geq t_0$ and inequality (\ref{eq11}) becomes (for all $t>0$ because $f(t)$ is bounded in $[0, t_0]$):
 
 \begin{equation}\label{eq12}
f(t)\leq\frac{1}{\sqrt{-b}}\{ 2 \bar C_1\int_0^t f(s) e^{- \sqrt{-b} s} ds+\bar C_2+ \bar C_3 f(t)  \} 
 \end{equation} 
 
 Now, let us denote $y(t)= \int_0^t f(s) e^{- \sqrt{-b} s} ds$. Then, $y'(t) = f(t) e^{- \sqrt{-b} t}$ and $y(0)=0$. Therefore (\ref{eq12}) becomes the differential inequality:
 
  \begin{equation}\label{eq13}
A e^{ \sqrt{-b} t} y'(t)-By(t) \leq C
 \end{equation} 

\noindent with $A=1-\bar C_3$, $B=\frac{2 \bar C_1}{\sqrt{-b}} > 0$ and $C=\bar C_2 > 0$.\\
 Let us suppose that $A \neq 0$ (if $A=0$, then we have the result using (\ref{eq13})).\\ Then we now have the differential inequality
 \begin{equation}\label{eq13B}
 y'(t) \leq \frac{C}{A} e^{-\sqrt{-b}t} + \frac{B}{A} e^{-\sqrt{-b}t} y(t)=F(t,y(t))
 \end{equation}
 
 As $F(t)$ is continuous and locally Lipschitz, if we consider 
 $$u_0(t)= \frac{C}{B}(e^{\frac{B}{A}(1-e^{ -\sqrt{-b} t})}-1)$$
 \noindent  the solution of $y'(t)=F(t,y(t))$ with $y(0)=0$, by applying Theorem 1.4 in \cite{Hartman}, we have that for all $t >0$, 
 \begin{equation}
 y(t)=\int_0^t f(s) e^{- \sqrt{-b} s} ds \leq u_0(t)= \frac{C}{B}(e^{\frac{B}{A}(1-e^{ -\sqrt{-b} t})}-1) \leq C <\infty
 \end{equation}
 so now inequality (\ref{eq11}) becomes, 
 
 \begin{equation}\label{eq14}
f(t)\leq\frac{1}{\sqrt{-b}}\{ A_1 + A_2 \sqrt{f(t)} \} 
 \end{equation} 
 
 \noindent with $A_1=2\bar C_1 C+\bar C_2 > 0$ and $A_2 = \bar C_3> 0$
 
 Let us denote $g(t)= \sqrt{f(t)}$ and inequality (\ref{eq14}) becomes
 
  \begin{equation}\label{eq15}
g^2(t)-A_2 g(t)-A_1 \leq 0\qquad \forall t > 0 
 \end{equation} 
 
 Therefore, $g(t)$ lies between the zeroes of the function $x^2-A_2x -A_1$, which are real and distinct numbers, because $A_1 \geq 0$ and $A_2 \geq 0$, and it is not possible that $A_1= A_2=0$. Hence, $g(t)$ (and also $g^2(t)=f(t)=\frac{v(t)}{e^{ \sqrt{-b} t}}$) is bounded, so the Theorem is proven by using inequality (\ref{eq3}). 
 
\section{Proof of Theorem \ref{thmMain3}}

This proof is modeled on the proof of Theorem 3 in \cite{Ch2}. As $S$ is  minimal, we apply Theorem \ref{isopT}, the fact that the center of the extrinsic balls $o \in S$, and the co-area formula to obtain (see \cite{Pa} for detailed proof), that the function $v(t)= \Vol(E_t)$ satisfies
\begin{equation}\label{stp1}
v(t) \geq \Vol(B^{b,2}_t)\,\,\,\forall t >0
\end{equation}
Now, using the co-area formula again and the fact that the function $f(t)= 
\frac{\Vol(E_t)}{\Vol(B^{b,m}_t)}$ is monotone
non-decreasing in $t$ (and hence $v'(t) \geq \frac{2 \pi}{\sqrt{-b}}\sinh\sqrt{-b}t\,\,\,\forall t>0$), we have
\begin{equation}\label{stp2}
\int_{S} \frac{1}{\cosh^3\sqrt{-b} r} d\sigma \geq \frac{2\pi}{\sqrt{-b}}\int_0^\infty \frac{\sinh\sqrt{-b} t}{\cosh^3\sqrt{-b}} dt=\frac{\pi}{-b}
\end{equation}
As, on the other hand, 
\begin{equation}\label{stp3}
\lim_{t \to 0} \frac{\int_{E_t} \cosh r d\sigma}{\cosh^2 t} \leq \lim_{t \to 0} \frac{ v(t)}{\cosh t} =0
\end{equation}
by applying Proposition \ref{PropIneq}, we have:
\begin{equation}
\begin{aligned}
 \frac{\pi}{-b}=\lim_{t \to \infty} \frac{\int_{E_t} \cosh r d\sigma}{\cosh^2 t} &\geq \lim_{t \to \infty}\int_{E_{t}}\frac{1+\sinh^{2}\sqrt{-b}r\Vert\nabla^{\bot}r\Vert^{2}}{\cosh^{3}
\sqrt{-b}r}d\sigma\\ =\int_{S} \frac{1}{\cosh^3\sqrt{-b} r} d\sigma&+\int_S \frac{1+\sinh^{2}\sqrt{-b}r\Vert\nabla^{\bot}r\Vert^{2}}{\cosh^{3}
\sqrt{-b}r}d\sigma \\ \geq \frac{\pi}{-b}&+\int_S \frac{1+\sinh^{2}\sqrt{-b}r\Vert\nabla^{\bot}r\Vert^{2}}{\cosh^{3}
\sqrt{-b}r}d\sigma
\end{aligned}
\end{equation}
so $\int_S \frac{1+\sinh^{2}\sqrt{-b}r\Vert\nabla^{\bot}r\Vert^{2}}{\cosh^{3}
\sqrt{-b}r}d\sigma=0$ and hence $\Vert\nabla^{\bot}r\Vert=0$ on $S$. Therefore $\Vert\nabla r\Vert=1$ on $S$ and $S$ is a minimal cone in $N$. Moreover, by applying Theorem 3.1 in \cite{Mi}, $\chi(E_t)=\chi(S)$ for all $t >0$. As, for sufficiently small $t$, the extrinsic and the geodesic balls are diffeomorphic, $E_t \equiv B^{b,2}_t$, then $\chi(S)=1$.


\end{document}